\documentclass[12pt]{article}
\usepackage{amsfonts}
\usepackage{amssymb}
\usepackage{amsthm}

\def\r{\mathbb R}
\def\h{\mathbb H}
\def\s{\mathbb S}
\def\l{\mathbb L}

\def\im{\mathrm {Im}}

\newtheorem{theorem}{Theorem}[section]

\newtheorem{proposition}[theorem]{Proposition}
\newtheorem{remark}[theorem]{Remark}

\newtheorem{lemma}[theorem]{Lemma}

\begin{document}

\title{Spacelike surfaces  with free boundary in the Lorentz-Minkowski space}
\author{Rafael L\'{o}pez\\
Departamento de Geometr\'{\i}a y Topolog\'{\i}a\\
Universidad de Granada\\
18071 Granada, Spain\\
email: rcamino@ugr.es \\
\vspace*{.5cm}\\
Juncheol Pyo\\
 Department of Mathematics\\
Pusan National University\\
Busan 609-735,  Korea\\
email: jcpyo@pusan.ac.kr}

\date{}

 \maketitle

\begin{abstract}
We investigate a variational problem in the Lorentz-Minkowski
space $\l^3$ whose critical points are spacelike surfaces with
constant mean curvature and making constant contact angle with a
given support surface along its common boundary. We show that if
the support surface is a pseudosphere, then the surface is a
planar disc or a hyperbolic cap. We also study the problem of
spacelike hypersurfaces with free boundary in the higher
dimensional Lorentz-Minkowski space $\l^{n+1}$.
\end{abstract}



\section{Introduction}

It is well known that spacelike hypersurfaces with constant mean
curvature in Lorentzian spaces   are critical points of the area
functional under deformations that keep constantly the enclosed
volume  by the hypersurface.   The utility in Physics, specially in general relativity,   of constant mean curvature spacelike hypersurfaces is that they are convenient as initial data for the Cauchy problem of the Einstein equations and  their uniqueness property (\cite{bf,ge}): in a cosmological spacetime, there exists at most one compact Cauchy hypersurface with a given (non-zero) constant mean curvature and a maximal (i.e., zero mean curvature) compact Cauchy hypersurface is almost unique. Such hypersurfaces are important because their properties reflect those of the spacetime, such as, for example, the proof of the positive mass conjecture (\cite{sy}). Other interested property is that if there exists one compact hypersurface with constant mean curvature in a cosmological spacetime, in a neighbourhood of that hypersurface there exists a foliation by compact hypersurfaces with constant mean curvature and the mean curvature varies in a monotone way from slice to slice. Thus a function $t$ can be defined by the property that its value at each point of a given leaf of the foliation is equal to the mean curvature of that leaf. If this foliation is global and in the case that $t$ is non-zero, then the function $t$   provides a global time coordinate, whose gradient  is a timelike vector field in the spacetime (\cite{ba,cfm,fmm,mt}).   This may be of relevance for the problem of time in quantum gravity (\cite{be}). In particular, the problem of existence of such hypersurfaces is central in this theory and it has been treated widely (we only refer \cite{bi} and \cite{re} and references therein).  Also, this type of hypersurfaces become
asymptotically null as they approach infinity and then they are
suitable for studying propagation of gravitational waves
(\cite{go,st}).

In this work we are interested in a modified version of the
variational problem. Consider the three-dimensional space $\l^3$.
Given a smooth region $W\subset\l^3$, we consider a compact
spacelike surface $M$  that is a critical point of the area among
surfaces in $W$ preserving the volume of $M$, the boundary lies on
$\partial W$ and the  interior is included in  the interior of
$W$.  Then we admit that the deformations $M_t$ of $M$ are
subjected to the constraint that all boundaries $\partial M_t$
move in the prescribed support surface $\Sigma=\partial W$.
Besides the area of the surface,  in the energy functional we have
to add a term that represents, with a certain weight $\lambda$,
the energy (area) of the part $\Omega_t$ of $\Sigma$ bounded by
$\partial M_t$. In particular, the induced metric in $\Sigma$ is
non-degenerate. In this context, the critical points of the energy
for any volume-preserving variation are called stationary surfaces
and they are characterized by two properties: i) the mean
curvature is constant and ii) the contact angle of the surface
with the support $\Sigma$ along its boundary is also constant.
Because $\Sigma$ is non-degenerate, it makes sense to consider the
angle between the unit normal vectors of the stationary surface
$M$ and support $\Sigma$. In Euclidean space, this type of
problems appear in the context of the theory of capillary
surfaces, with a great influence in many areas of physics and
chemistry.

In Lorentzian spaces, this problem is introduced by Al\'{\i}as and Pastor (\cite{ap2})  considering the Lorentz-Minkowski space as the ambient space and  a spacelike plane or a hyperbolic plane as support surface $\Sigma$. When the ambient spaces are the other simply connected Lorentzian space forms, namely, the de Sitter space and the
anti-de Sitter space, similar results have been obtained in \cite{pa}. A first kind of supports are the umbilical surfaces of $\l^3$, that is,  non-degenerate planes, hyperbolic planes and pseudospheres (\cite{on}). As the boundary of a spacelike compact surface is a closed spacelike curve, in the case that  the support surface is a plane, this  plane must be spacelike.

This article is motivated by the results that appear in
\cite{ap2}. Among the umbilical surfaces of $\l^3$, the remaining
case to study is that $\Sigma$ is a pseudosphere, which is
considered in this work, showing:

\begin{theorem}\label{t1}
 The only stationary spacelike surfaces in $\l^3$ with embedded connected boundary and with pseudosphere as support surface are the planar discs   and the hyperbolic caps.
\end{theorem}

We point out that  it is implicitly assumed in the results  of
\cite{ap2,pa} that the boundary of the stationary surface is
embedded and connected. For example, in $\l^3$ there are pieces of
rotational spacelike surfaces with constant mean curvature bounded
by two concentric circles contained in the same plane and
orthogonal to the rotational axis. Of course, these surfaces
satisfy the boundary condition on the contact angle (see pictures
in \cite{lo2}). More recently, the first author has studied stationary surfaces in $\l^3$ with some assumption on the symmetry  of the surface, relating geometric quantities of the surface such as its height,  area and volume: \cite{lo3,lo4}.

The paper is organized as follows. After a preliminaries section,
where we present the variational problem, we consider in Section
\ref{sect-pseudo} our specific setting when the support is a
pseudosphere. The fact that the boundary of stationary surface is a
spacelike curve and a pseudosphere is timelike
 makes  necessary a slight modification of the
variational problem, which will be reformulated in this section.
In Section \ref{sect-main} we prove Theorem \ref{t1} and we finish
in Section \ref{sect-finish} with some discussions of the
variational problem in the general $n$-dimensional case.

\section{Preliminaries}\label{sect-preli}

Let $\l^3$ be the three dimensional
Lorentz-Minkowski space, that is, the real vector space
$\mathbb{R}^3$ endowed with the Lorentz-Minkowski metric
$$\langle,\rangle = (dx_1)^2 +(dx_2)^2 - (dx_3)^2,$$
where $(x_1, x_2, x_3)$ are the canonical coordinates of $\mathbb{R}^3$.  Given a connected surface $M$, a smooth immersion $x:M\rightarrow\l^3$ is called spacelike if the induced metric on $M$ via $x$ is positive definite. Observe that $a=(0,0,1)$ is a unit timelike vector field globally defined on $\l^3$, which determines a timelike orientation on $\l^3$. Thus, given a spacelike immersion, we can choose  a unique unitary timelike normal field $N$ globally defined on $M$ such that $\langle N,a\rangle<0$. This shows that $M$ is orientable. If $N$ is chosen as above, we say that $N$ is future-directed and the $M$ is oriented by a unit future-directed timelike normal vector field $N$.

Among examples of spacelike surfaces in $\l^3$, we point out here
the totally umbilical ones, that is, spacelike planes and hyperbolic
planes. A spacelike plane is given by the set $\{x\in\l^3;\langle
x-p,v\rangle=0\}$, where $p\in\l^3$ and $v$ is a timelike vector of
$\l^3$. The mean curvature is $H=0$. After an isometry, a hyperbolic
plane $\h^2(p,r)$ centered at $p$ and radius $r>0$ is given by
$$\h^2(p,r)=\{x\in\l^3;\langle x-p,x-p\rangle=-r^2\}.$$
This surface has two components: $\h^2_{+}(p,r)=\{x\in\h^2(p,r);x_3\geq x_3(p)\}$ and $\h^2_{-}(p,r)=\{x\in\h^2(p,r);x_3\leq -x_3(p)\}$. With the future-directed timelike orientation, the mean curvatures of $\h^2_{+}(p,r)$ and $\h^2_{-}(p,r)$ are  $1/r$ and $-1/r$, respectively.

We state the variational problem. Here we follow \cite{ap2} and we refer there for more details. Let $\Sigma$ be an embedded connected non-degenerate surface in $\l^3$ that divides the ambient space $\l^3$ in two connected components denoted by  $\l^3_+$ and $\l^3_{-}$. Let $M$ be a compact surface. In all results of this work, it is assumed that the boundary of $M$ is connected, although this is not necessary to establish  the variational problem.
Given a spacelike immersion $x:M\rightarrow\l^3$, an admissible variation of $x$ is a smooth map
$X:M\times (-\epsilon,\epsilon)\rightarrow\l^3$ such that for each $t\in (-\epsilon,\epsilon)$, the map $X_t:M\rightarrow \l^3$ defined by $X_t(p)=X(p,t)$ is a spacelike immersion with
$$X_t(\textrm{int}(M))\subset \l^3_+\ \ \mbox{and}\ \ X_t(\partial M)\subset\Sigma$$
and at the initial time $t=0$, we have $X_0=x$. The surface $\Sigma$ is called the support surface. Let $\lambda\in\r$. The energy function $E:(-\epsilon,\epsilon)\rightarrow\r$ is defined by
\begin{equation}\label{energy}
E(t)=\int_M dA_t+\lambda \int_{\Omega_t} d\Sigma
\end{equation}
where $\Omega_t\subset\Sigma$ is the domain bounded by
$X_t(\partial M)$,  $dA_t$ denotes the area element of $M$ with
respect to the metric induced by $X_t$ and $d\Sigma$  is the area
element on $\Sigma$.

The volume function of the variation $V:(-\epsilon,\epsilon)\rightarrow\r$ is defined by
$$V(t)=\int_{M\times[0,t]}X^*(dV),$$
where $X^*(dV)$ is the pullback of the canonical volume element $dV$
of $\l^3$. The variation is said to be volume-preserving if
$V(t)=V(0)=0$ for all $t$. The expressions of the first variation
formula for the energy $E$ and the volume $V$ are as follows:
\begin{equation}\label{var-area}
E'(0)=-2\int_M H\langle N,\xi\rangle dA-\int_{\partial M}\Big(\langle\nu,\xi\rangle+\lambda\langle\nu_\Sigma,\xi\rangle \Big) ds
\end{equation}
\begin{equation}\label{var-volume}
 V'(0)=-\int_M\langle N,\xi\rangle dA.
 \end{equation}
Here $N$ is an orientation on $M$, $H$ the corresponding mean curvature function, $\nu$ and $\nu_\Sigma$ are the unit inward conormal vectors of $M$ and $\Omega$ along $\partial M$, respectively, and $\xi$ is the variation vector field of the variation $X$:
$$\xi(p)=\frac{\partial X}{\partial t}(p,0).$$
Consider $\tau$ a unit tangent vector to $\partial M$ and let $\{\tau,\nu,N\}$ and $\{\tau,\nu_\Sigma,N_\Sigma\}$ be two orthonormal bases such that $\textrm{det}(\tau,\nu,N)=\textrm{det}(\tau,\nu_\Sigma,N_\Sigma)=1$. Denote $\epsilon=1$ or $-1$ depending   if $\Sigma$ is timelike or spacelike, that is, $\langle N_\Sigma,N_\Sigma\rangle=\epsilon$. Recall that as $\partial M$ is spacelike, then $\langle\nu_\Sigma,\nu_\Sigma\rangle=-\epsilon$. With respect to $\{\nu_\Sigma,N_\Sigma\}$, $N$ and $\nu$ are given by
\begin{eqnarray}\label{nuN}
N&=&-\epsilon\langle N,\nu_\Sigma\rangle\nu_\Sigma+\epsilon\langle N,N_\Sigma\rangle N_\Sigma.\\
 \nu&=&-\langle N,N_\Sigma\rangle\nu_\Sigma+\langle N,\nu_\Sigma\rangle N_\Sigma.
 \end{eqnarray}
Thus
$$\langle\nu,\xi\rangle=\langle N,\nu_\Sigma\rangle\langle N_\Sigma,\xi\rangle-\langle N,N_\Sigma\rangle\langle\nu_\Sigma,\xi\rangle.$$
Then (\ref{var-area}) writes as
\begin{eqnarray}\label{energy2}
E'(0)&=& -2\int_M H\langle N,\xi\rangle dA-\int_{\partial M}(\lambda-\langle N,N_\Sigma\rangle)\langle\nu_\Sigma,\xi\rangle ds\\
&-&\int_{\partial M} \langle N,\nu_\Sigma\rangle\langle N_\Sigma,\xi\rangle ds.\nonumber
\end{eqnarray}
The last term of (\ref{energy2}) vanishes since the vector field
$\xi$ is tangential to $\Sigma$
 along the boundary $\partial M$. Indeed, if $p\in\partial M$, the curve $t\longmapsto X(p,t)$ lies in the support surface
$\Sigma$ and then its velocity is  tangent to $\Sigma$. But at $t=0$, this velocity is just $\xi(p)$. Thus $\xi(p)\in T_p\Sigma$ and $\langle N_\Sigma,\xi\rangle=0$ along $\partial M$.

We say that the immersion $x$ is stationary if $E'(0)=0$ for every
admissible volume-preserving variations of $x$. By the method of
Lagrange multipliers, there exists $\mu\in\r$ such that $E'(0)+\mu
V'(0)=0$ for any such variations. From (\ref{var-volume}) and
(\ref{energy2}) we have
$$\int_M(2H+\mu) \langle N,\xi\rangle dA+\int_{\partial M}(\lambda-\langle N,N_\Sigma\rangle)\langle\nu_\Sigma,\xi\rangle ds=0.$$
A standard argument gives:

\begin{theorem}\label{t-cara} In the above conditions, the immersion $x:M\rightarrow \l^3$ is stationary if and only if the mean curvature $H$ is constant and the angle between $M$ and $\Sigma$ along $\partial M$  is constant.
\end{theorem}

In case $\Sigma$ is spacelike,   $N_\Sigma$ is a unit timelike
vector. Considering   both $N$ and $N_\Sigma$ are oriented by a
unit future-directed timelike orientation,  the angle $\theta$ is
defined as $\lambda=\langle N,N_\Sigma\rangle=-\cosh\theta$. If
$\Sigma$ is a timelike surface, $N_\Sigma$ is a unit spacelike
vector. Assuming again that $M$ is future-directed oriented, the
angle between $M$ and $\Sigma$ is the number $\theta$ such that
$\langle N,N_\Sigma\rangle=\sinh\theta$ (\cite{Ratcliffe}).

\begin{remark}
\begin{enumerate}
\item When $\lambda=0$, we have the classical problem of a
    surface with critical area and with free boundary in
    $\Sigma$. In such a case, the intersection between $M$ and
    $\Sigma$ is orthogonal.
\item Our   definition of the energy $E$ in (\ref{energy}) is
    motivated by what occurs in Euclidean setting when one
    considers liquid drops resting in some support. In order
    to define $E(t)$ and $V(t)$, however, it is not necessary
    that the images of the immersions $X_t$ of the variation
    lie in one of the two domains determined by $\Sigma$.
\item In general, we call a stationary surface $M$ in $\l^3$ supported on a non-degenerate surface $\Sigma$ as a spacelike surface with constant mean curvature whose boundary lies in $\Sigma$ and $M$ and $\Sigma$ make constant contact angle along the boundary of $M$.
\end{enumerate}
\end{remark}

\section{The case of pseudosphere as support surface}\label{sect-pseudo}

After an isometry of $\l^3$, a pseudosphere   $\s^2_1(p,r)$
centered at $p\in\l^3$ and radius $r>0$, is defined by
$$\s^2_1(p,r)=\{x\in\l^3;\langle x-p,x-p\rangle=r^2\}.$$
Recall that $a=(0,0,1)$. This surface is timelike  with constant
curvature $1/r^2$. Denote the  waist of $\s^2_1(p,r)$ as
$C(p,r)=\s^2_1(p,r)\cap\{x_3=\langle p,a \rangle \}$.

Let  $x:M\rightarrow\l^3$ be a compact spacelike surface immersed
into $\l^3$.    For a given closed curve
 $\Gamma\subset\s^2_1(p,r)$, we say that $\Gamma$ is the boundary of
 the immersion $x$ if the restriction map $x_{|\partial M}:\partial M\rightarrow\Gamma$ is a diffeomorphism.
Because our problem is invariant by homotheties and isometries of
the ambient space, without loss of generality we will assume that
the support surface is the unit pseudosphere centered at the origin
$O$, that is, $\s^2_1(O,1)=\s^2_1=\{x\in\l^3;\langle x,x\rangle=1\}$
with $C(O,1)=C$ as its waist. This surface is also known in the
literature as the de Sitter surface of $\l^3$. The ambient space
$\l^3$ is divided by $\s^2_1$ into two domains $\l^3_+$ and
$\l^3_-$:
$$\l^3_+=\{x\in\l^3;\langle x,x\rangle<1\},\ \ \l^3_-=\{x\in\l^3;\langle x,x\rangle>1\}.$$
The pseudosphere $\s^2_1$ can be globally parametrized by means of a diffeomorphism  $F:\r\times C\rightarrow\s^2_1$ given by $F(t,q)=\gamma_q(t)$, where
$$\gamma_q(t)=\mbox{exp}_q(ta)=\cosh(t)q+\sinh(t)a,$$
is the (future pointing) unitary geodesic orthogonal to $C$ through the point $q\in C$. Recall that $a=(0,0,1)$. The orthogonal projection $\pi:\s^2_1\rightarrow C$ associates to each $p\in\s^2_1$ the unique point $\pi(p)$ such that $F(t,\pi(p))=p$ for a certain $t$. By the expression of $F$, we deduce
\begin{equation}\label{expressionpi}
\pi(p)=\frac{1}{\sqrt{1+\langle p,a\rangle^2}}\Big(p+\langle p,a\rangle a\Big),\ p\in\s^2_1.
\end{equation}
If $u$ is a smooth function defined on $C$, the geodesic graph of
$u$ (on $\s^2_1$) is the curve given by $\{F(u(q),q);q\in\s^1\}$.
If $p\in \s^2_1$ and $v\in T_p\s^2_1$, a straightforward
computations gives
\begin{equation}\label{pi}
(d\pi)_p(v)=\frac{v-\langle v,a\rangle a}{\sqrt{1+\langle p,a\rangle^2}}-
\frac{\langle p,a\rangle\langle v,a\rangle}{(1+\langle p,a\rangle)^{\frac{3}{2}}}(p+\langle p,a\rangle a).
\end{equation}

\begin{proposition}\label{pr-graph} Let $\alpha:\s^1\rightarrow \s^2_1$ be a closed spacelike curve. Then $\pi\circ\alpha$ is a covering map of $C$. In particular, if $\alpha$ is an embedding then $\alpha(\s^1)$ is a geodesic graph on $C$. In particular, in  $\s^2_1$ there exist no spacelike nulhomologous curves.
\end{proposition}

\begin{proof} Define $\psi=\pi\circ\alpha:\s^1\rightarrow C$, that is,
$$\psi(s)=\frac{1}{\sqrt{1+\langle\alpha(s),a\rangle^2}}\Big(\alpha(s)+\langle\alpha(s),a\rangle a\Big).$$
From (\ref{pi}),
$$\langle\psi'(s),\psi'(s)\rangle=\frac{\langle\alpha'(s),\alpha'(s)\rangle}{1+\langle\alpha(s),a\rangle^2}.$$
Because $\alpha$ is spacelike, the map $\psi$ is a local diffeomorphism and hence $\psi:\s^1\rightarrow C$ is a covering map. When $\alpha$ is an embedding, then the covering map $\psi$ is one-to-one, that is, $\psi$ is a global diffeomorphism between $\s^1$ and  $C$, showing that $\alpha$ is a graph on $C$: $\pi_{|\alpha(\s^1)}=\psi\circ\alpha^{-1}$.
\end{proof}

\begin{remark} In the higher dimensional case $(n\geq2)$,  if $M^n$ is a compact submanifold and $x:M^n\rightarrow\s^{n+1}_1$ is a spacelike hypersurface in the $(n+1)$-dimensional de Sitter space $\s^{n+1}_1$, the map  $\psi=\pi\circ x$ is a covering map between $M^n$ and the $n$-sphere $\s^n=\s^{n+1}_1\cap\{x_{n+2}=0\}$, which is simply-connected. Thus the covering map $\psi$ is a one-to-one (\cite{mo}).
\end{remark}

From Proposition \ref{pr-graph}, if $x:M\rightarrow\l^3$ is a spacelike
immersion of a compact surface $M$, with $x(\partial M)=\Gamma$ a
closed curve included in $\s^2_1$, then $\Gamma$ does not bound a
domain of $\s^2_1$ and thus, the variational problem established
in Section \ref{sect-preli} should be reformulated. With this
purpose, let $\Gamma$ be a spacelike embedded  curve in $\s^2_1$
and consider the waist $C$ of $\s^2_1$. Since $\Gamma$ is
homologous to $C$, $\Gamma\cup C$ bounds a domain
$\Omega\subset\s^2_1$. Given a variation  $X$ of $x$, for values
$t$ closes to $t=0$, $X_t(\partial M)$ is homologous to $\Gamma$.
Then in the second term of $E$ in (\ref{energy}), we replace
$\Omega_t$ by the domain bounded by $X_t(\partial M)\cup C$. Let
us observe that one can change $C$ for other curve $C'$ homologous
to $C$ because the corresponding integral
$\int_{\Omega_t'}d\Sigma$ changes only by an additive constant,
with no consequence in the formula $E'(0)$. Therefore, in the case
that $\Sigma$ is a pseudosphere, Theorem \ref{t-cara} holds in the
same terms.

\begin{remark} We point out  that there are no stationary surfaces included in $\l^3_{-}$ and with embedded boundary $\Gamma$ in $\s^2_1$, because  $\Gamma$ must be homologous to $C$, but
there are no compact   surfaces in $\l^3_{-}$ spanning $\Gamma$, independently if the surface is or it not spacelike.
\end{remark}

\section{Proof of Theorem \ref{t1}}\label{sect-main}

From now on, we consider that the boundary $\Gamma$  of the
stationary surface $M$ is an embedded connected spacelike curve in
$\s^2_1$ and thus, homologous to $C$. Replacing $C$ by other
homologous curve $C'$ if necessary, we can assume that the domain
$\Omega\subset\s^2_1$ bounded by $\Gamma\cup C$ is an embedded
surface. Denote $\nu_S$ the unit inward conormal vector of
$\Omega$ along $\Gamma$ and $N_S$ is the Gauss map of $\Omega$. In
particular, for $p\in\Omega$,  $N_S(p)=p$. Let $\theta$ be the
constant such that $\langle N,N_S\rangle=\sinh\theta$. From
  (\ref{nuN}) we have
\begin{eqnarray}
N&=&-\cosh\theta\nu_S+\sinh\theta  N_S.\label{nuN2}\\
\nu&=&-\sinh\theta \nu_S+\cosh\theta  N_S.\label{nuN22}
\end{eqnarray}

From Proposition \ref{pr-graph}, we have

\begin{lemma}
 Let $\Gamma$ be a closed embedded spacelike curve in $\s^2_1$. Then $\Gamma$ is a graph on the plane $P= \{x_3=0\}$. Moreover, the orthogonal projection of $\Gamma$ on $P$ bounds a simply-connected domain.
\end{lemma}

\begin{proof}
  By Proposition \ref{pr-graph}, $\Gamma$ is a graph of $\s^2_1$ on $C$. Let $\Pi:\l^3\rightarrow P$ be the orthogonal projection   onto   $P$, that is, $\Pi(q)=q+\langle q,a\rangle a$. From (\ref{expressionpi}) we obtain
  $\Pi(q)=\sqrt{1+{\langle q,a\rangle}^2}\pi(q)$. On $\Gamma$, the map $\Pi:\Gamma\rightarrow P$ is a local diffeomorphism since $\Gamma$ is spacelike: if $\alpha:\s^1\rightarrow\Gamma$ is a parametrization, $\alpha=\alpha(s)$,   we have
  $$|(\Pi\circ\alpha)'(s)|^2=|\alpha'(s)|^2+\langle\alpha'(s),a\rangle^2\geq |\alpha'(s)|^2.$$
On the other hand, if  there exist two distinct points $q_1, q_2\in\Gamma$ such that $\Pi(q_1)=\Pi(q_2)$,  the symmetry of $\s^2_1$  implies that $q_2=-q_1$. From the above relation between $\Pi$ and $\pi$, $\pi(q_1)=\pi(q_2)$: contradiction. Thus $\Pi:\Gamma\rightarrow\Pi(\Gamma)$ is a diffeomorphism.
\end{proof}

\begin{lemma}\label{lemma2}
Let $x:M\rightarrow\l^3$ be a compact spacelike immersion whose boundary $\Gamma$ is an embedded curve included  $\s^2_1$. Then $x(M)$ is a graph on $P$ and thus, $x(M)$ is a topological disc.
\end{lemma}

\begin{proof}From the above lemma, there exists a simply-connected compact domain $D\subset P$ such that  $\Gamma$ is a graph on $\partial D$. If $\Gamma$ is the boundary of a spacelike immersed surface $M$, then it is known that $M$ is a graph on $D$ and thus, a topological disc: see for example, \cite[Lemma 3]{ap2}. The idea is the following. As in the above proof, the fact that $M$ is spacelike means that $\Pi:x(M)\rightarrow P$ is a local diffeomorphism with
$|(d\Pi)_p(v)|^2=|v|^2+\langle v,a\rangle^2\geq |v|^2$. The spacelike condition on $M$ implies that $\Pi(M)=\overline{D}$, in particular $\Pi:x(M)\rightarrow \overline{D}$ is a covering map. As $D$ is a simply-connected domain, then  $\Pi:M\rightarrow \overline{D}$ is a global diffeomorphism.
\end{proof}

We now proceed to show Theorem \ref{t1}.

\textit{Proof of Theorem \ref{t1}} Let $x:M\rightarrow\l^3$ be a
stationary surface. By Lemma \ref{lemma2} we can parametrize $M$
conformally to the closed unit disc $\overline{{\mathbb
D}}=\{(u,v)\in\r^2;u^2+v^2\leq 1\}$ such that $z=u+iv$ is a
conformal coordinate of $\overline{{\mathbb D}}$. By conformality,
we have
$$\langle x_u,x_u \rangle=\langle x_v,x_v \rangle=E^2,\ \langle x_u,x_v \rangle=0.$$
Let $h_{ij},$ $i,j=1,2$ be the coefficients of the second fundamental form of $x:M\rightarrow\l^3$. More precisely,
$$h_{11}=\langle x_{uu},N \rangle,\ h_{12}=h_{21}=\langle x_{uv},N \rangle,\ h_{22}=\langle x_{vv},N \rangle.$$
We introduce the Hopf quadratic differential $\phi=\phi(z,\bar{z})=(h_{11}-h_{22}-2ih_{12})dz^2$ which is invariant by a conformal coordinate of $M.$ The Hopf differential $\phi$ has two important properties:
\begin{enumerate}
\item $\phi$ is holomorphic if and only if the mean curvature of the immersion is constant. This is a consequence of the Codazzi equation.
\item $\phi$ vanishes at some point $p\in M$ if and only if $p$ is an umbilical point. This occurs because  $|h_{11}-h_{22}-2ih_{12}|^2=\frac{1}{4E^2}(H^2+K)\geq 0$,  where $K$ is the Gaussian curvature of $M$.
\end{enumerate}
As a consequence, in a constant mean curvature spacelike surface the holomorphicity of $\phi$ implies that the set of  umbilical points coincides with the zeroes of a holomorphic differential form. Therefore, either umbilical points are isolated or the immersion is totally umbilical.

On the boundary $\partial M$, we have $|z|=1$ and then,
$z=e^{i\theta}$. Since
$\partial_{z}=\frac{1}{2}(\partial_u-i\partial_v)=\frac{1}{2}(\bar{z}\partial_r-i\bar{z}\partial_{\theta})$,
then
$$\phi=4\sigma(\partial_z,\partial_z)dz^2=  \bar{z}^2\sigma(\partial_r,\partial_r)-2i\bar{z}^2\sigma
 (\partial_r,\partial_{\theta})-\bar{z}^2\sigma(\partial_{\theta},\partial_{\theta}).$$

 Hence on $|z|=1$, we get
 $$\im(z^2\phi)=-\sigma(\partial_r,\partial_{\theta}).$$
 On the other hand, the unit tangent $t$ and the inward-pointing unit conormal $\nu$ along $\partial M$ are
denoted by $$t=E^{-1}\partial_{\theta},\ \ \nu=-E^{-1}\partial_{r}.$$
We have then $\im(z^2\phi)=E^2\sigma(t,\nu)$.
If $\overline{\nabla}$ is the Levi-Civita connection on $\l^3$, using (\ref{nuN2})-(\ref{nuN22}) and $N_S(p)=p$, we obtain:
\begin{eqnarray*}\sigma(t,\nu)&=&
-\langle\overline{\nabla}_t N,\nu\rangle=\langle \overline{\nabla}_t \nu,N\rangle=\langle -\sinh\theta\overline{\nabla}_t\nu_S+\cosh\theta  \overline{\nabla}_t N_S,N\rangle\\
&=&-\sinh\theta\langle \overline{\nabla}_t\nu_S,N\rangle+\cosh\theta\langle t,N\rangle=-\sinh\theta\langle \overline{\nabla}_t\nu_S,N\rangle\\
&=&\sinh\theta\cosh\theta\langle \overline{\nabla}_t\nu_S,\nu_S\rangle-
\sinh^2\theta\langle \overline{\nabla}_t\nu_S,N_S\rangle\\
&=&\frac12 \sinh\theta\cosh\theta  t\langle\nu_S,\nu_S\rangle+\sinh^2\theta\langle\nu_S,\overline{\nabla}_t N_S\rangle\\
&=&
\sinh^2\theta\langle\nu_S,t\rangle=0.
\end{eqnarray*}
In other words, the harmonic function $\im(z^2\phi)$ vanishes on
$\partial \mathbb D$, hence it must be identically zero in
$\mathbb D$. This implies that the holomorphic function $z^2\phi$
must be constant in $\mathbb D$. Since at origin the value of
$z^2\phi$ is zero, then $z^2\phi\equiv0$. This implies that
$\phi=0$ on $M$ and the immersion is totally umbilical.  $\Box$

We remark that any spacelike plane intersects $\s_1^2$ at constant angle and that when a hyperbolic plane intersects $\s_1^2$, this occurs at constant contact angle too. Indeed, if $M$ is the plane $M=\{x\in\l^3;\langle x-p,v\rangle=0\}$, with $\langle v,v\rangle=-1$, $\langle v,a\rangle<0$, then for any $x\in M\cap\s_1^2$, we have
$N(x)=v$ and $N_S(x)=x$. Thus along $\partial M$, $\langle N, N_S\rangle=\langle p,v\rangle$. In the case that $M$ is a hyperbolic plane of type $\h^2_+(p,r)$ or $\h^2_{-}(p,r)$, $N(x)=(x-p)/r$ and then
$\langle N,N_S\rangle=(1-r^2-|p|^2)/(2r)$ along $\partial M$, hence that the contact angle is  constant again.

\section{Further discussions of the problem in arbitrary dimensions}\label{sect-finish}

In this section we give some remarks about the problem of hypersurfaces with free boundaries in the $(n+1)$-dimensional Lorentz-Minkowski space $\l^{n+1}$. Following \cite{ap2}, we ask whether the totally umbilical hypersurfaces are the only stationary hypersurfaces in $\l^{n+1}$ whose support hypersurface is an umbilical hypersurface.

In this sense, the conjecture 1 in \cite[p. 1330]{ap2} is true.
The proof combines the maximum principle and the characterization
of constant mean curvature compact spacelike  hypersurfaces
bounded by an $(n-1)$-sphere.

\begin{theorem}\label{t-conj1} The only stationary hypersurfaces with whose boundary is embedded and resting in a spacelike hyperplane are hyperplanar balls   and hyperbolic caps.
\end{theorem}

\begin{proof} Let $x:M^n\rightarrow\l^{n+1}$ be a spacelike immersion of a compact submanifold $M$ with connected boundary. Let $\Gamma=x(\partial M)$ and assume that $\Gamma$ is included in the spacelike hyperplane $P=\{x_{n+1}=0\}$. Because  $\Gamma$ is embedded, $\Gamma$ encloses a simply-connected domain $D\subset P$. The spacelike condition of the immersion implies that  the orthogonal projection $\Pi:M\rightarrow P$ is a local diffeomorphism. A reasoning similar as in \cite[Lemma 3]{ap2} shows that $x(M)$ is a graph on $D$, in particular, $x:M\rightarrow\l^{n+1}$ is an embedding. Using the maximum principle for the constant mean curvature equation (\cite{gt}), it is known that a graph is included in one side of $P$. Without loss of generality, we assume that $M$ lies over $P$, that is, $M-\partial M\subset \{x_{n+1}>0\}$.

As a consequence, the surface $M$ together $D$ encloses a domain
$W\subset\l^{n+1}$. Now we are in conditions to apply the
Alexandrov method by family of parallel vertical hyperplanes
(\cite{al}). The fact that the angle between $M$ and $P$ is constant
along $\Gamma$ makes that  the Alexandrov process works well
because if there is a contact point between boundary points, the
condition of the constant contact angle gives that the tangent
hyperplanes agree between a point of $\Gamma$ and its reflected
one: see \cite{lo} for an example in a more general context. Then
one shows that $M$ is rotationally symmetric with respect to a
straight-line $L$ orthogonal to $P$. In particular, the boundary
$\Gamma$ is a round $(n-1)$-sphere $\s^{n-1}$. In \cite{ap1} (see
\cite{alp} in the two-dimensional case) it is proved  that the
only compact spacelike hypersurfaces in $\l^{n+1}$ spanning
$\s^{n-1}$ are umbilical, showing the result.
\end{proof}

In the proof we have first showed that the graph lies in one side
of the hyperplane. In this part of the proof, it is not necessary
to use that $M$ is a graph and that $H$ is constant on $M$ but
that $H$ does not vanish on the surface.

\begin{theorem}\label{t-conj2} Let $x:M^n\rightarrow\l^{n+1}$ be a spacelike immersion of a compact manifold  $M$  whose boundary lies in a hyperplane $P$. If the (non-necessary constant) mean curvature $H$ does not vanish, then $M$ lies in one side of $P$. If $H=0$ on $M$, then $M$ lies included in $P$.
\end{theorem}

\begin{proof}
First, we point out that we do not know that $M$ is a graph or
not, since we admit the possibility that $x(\partial M)$ is not an
embedding. But we know that $P$ is a spacelike hyperplane because
the immersion is spacelike and $x(\partial M)$ is a closed
submanifold of $P$. Without loss of generality, we suppose that
$P$ is the hyperplane $x_{n+1}=0$. Denote $P_t=\{x\in
\l^{n+1};\langle x,a\rangle=-t\}$. As $M$ and $P_t$ are spacelike,
we consider the future-directed timelike orientations  $N$ and
$N_t$ respectively, that is, if $a=(0,\ldots,0,1)$, $\langle
N,a\rangle<0$ on $M$ and  $N_t= a$ for   any $t$.

Assume that $H\not=0$ on $M$. By contradiction, we assume that $M$ has points in both sides of $P$. At the highest point $p$ of $M$ with respect to the plane $P$, let us place  $P_{t_1}$, where   $t_1=x_{n+1}(p)>0$. As $N_{t_1}(p)=N(p)=a$ and $P_{t_1}$ lies above $M$, the maximum principle says that $0>H(p)$.  Because $H\not=0$, then we conclude that $H<0$  on $M$. Similarly, at the lowest point $q$, with $t_2=x_{n+1}(q)<0$, we place   the hyperplane $P_{t_2}$. Now $M$ lies over $P_{t_2}$ and the maximum principle implies $H(q)>0$: contradiction.

If $H=0$, the same reasoning as above says us that $M\subset P$.
\end{proof}

We try to do the same reasoning for stationary hypersurfaces being
the support hypersurface a hyperbolic hyperplane $\h^n$. Following
the same ideas as in Theorem \ref{t-conj1}, the two ingredients are i)
show that $M$ lies in one side of $\h^n$ and ii) use the
Alexandrov method to prove that $M$ is rotational, finishing as in
Theorem \ref{t-conj1}.

Here we prove the first step but, after showing this fact,  we are not able to   apply in a suitable way the Alexandrov process, because reflections with respect to hyperplanes do not leave invariant the support hypersurface $\h^n$.

\begin{theorem} Consider $M^n$ a compact $n$-dimensional manifold with non-empty boundary and $x:M^n\rightarrow\l^{n+1}$ be  a spacelike immersion with (non-necessary constant) mean curvature $H$. Assume that $x(\partial M)$ lies in a hyperbolic plane $\h^n$ and denote by $h>0$ the mean curvature of $\h^n$ for an appropriate orientation. If   $|H|\not=h$, then $M$ lies in one side of $\h^n$. If $|H|=h$, then either $M$ lies included in $\h^n$ or $M$ lies in one side of $\h^n$.
\end{theorem}

\begin{proof} We use notations similar as in Preliminaries in the context $\l^{n+1}$. Without loss of generality, we assume that the support hypersurface is $\h^n=\h^n_{+}(O,r)$, $O$ the origin of coordinates, $r=1/h$. We consider on $\h^n$ the future-directed timelike orientation $N_h$, that is, $N_h(p)=p$. Consider the foliation of $\l^{n+1}$ given by
$\h^n_+(ta,r)$, $t\in\r$. Let us remark that
$\h_{+}^n(0a,r)=\h^n$. All these hypersurfaces have constant mean
curvature $h>0$ for the future-directed timelike orientation. Each
one of the sides of $\h^n$ are $\{x\in\l^{n+1};\langle
x,x\rangle<-r^2\}$ and $\{x\in\l^{n+1};\langle x,x\rangle>-r^2\}$.

Suppose that $|H|\not=h$ and by contradiction, we assume that $M$
has points in both sides of $\h^n$. By the compactness of $M$, for
$t$ sufficiently close to $-\infty$,  $\h^n_{+}(ta,r)$ does not
intersect $M$. Letting $t \nearrow 0$, there is a first time $t_1<0$
such that $\h^n_{+}(t_1a,r)$ touches $M$ at a (interior) point $q$
but $\h^{n}_{+}(ta,r)\cap M=\emptyset$ for $t<t_1$. Let us compare
$M$ and $\h^n_{+}(t_1a,r)$ at $q$ and we use  that the orientations
of both hypersurfaces agree at $q$ since both are the
future-directed orientations. The maximum principle says that
$H(q)>h$. As $|H|\not= h$, then $H>h$ on $M$.

Consider now the other side of $\h^n$. By a similar reasoning, there
is $t_2>0$ such that $\h^n_{+}(t_2a,r)$ touches $M$ at a (interior)
point $p$ but $\h^{n}_{+}(ta,r)\cap M=\emptyset$ for $t>t_2$.
Comparing $M$ and $\h^n_{+}(t_2a,r)$ at $p$, the maximum principle
says now that $h>H(p)$: contradiction.

In the case that $|H|=h$ and assuming that $M\not\subset \h^n$, the same reasoning implies in the first step that $H(q)>h$, which is not possible. Thus the case $t_1<0$ is not possible. This shows that $M$ lies over $\h^n$. If the number $t_2>0$ exists, then we have that $H=-h$.
\end{proof}

The last situation in the proof appears when one consider $\h^n$
as above and $M$ is a  hyperbolic cap of   $\h^n_{-}(ta,r)$, for
$t>0$ sufficiently big.

\vspace*{.5cm}

Acknowledgment: The first author is partially supported by MEC-FEDER grant no. MTM2011-22547 and Junta de Andaluc\'{\i}a grant no. P09-FQM-5088. The second author   was supported by Basic Science Research Program through the National Research Foundation of Korea (NRF) funded by the Ministry of Education, Science and Technology (2011-0026779) and Pusan National University Research Grant, 2011.


\begin{thebibliography}{10}


 \bibitem{al}   Alexandrov A D 1962 \emph{Ann. Mat. Pura Appl.} \textbf{58}  303--315

\bibitem{alp}   Al\'{\i}as  L J  L\'opez R and  Pastor J A 1998 \emph{T\^{o}hoku Math. J. }\textbf{50}, 491--501

\bibitem{ap1} Al\'{\i}as L J and Pastor  J  A 1998 \emph{J. Geom. Phys.} \textbf{28}    85-–93

\bibitem{ap2} Al\'{\i}as L J and Pastor  J  A 1999 \emph{ Class. Quantum Grav.} \textbf{16}   1323--1331

\bibitem{ba} Bartnik  R 1984 \emph{Comm. Math. Phys.} \textbf{94} 155--175

\bibitem{bi} Bartnik R and Isenberg J 2004 \emph{SIAM J. Sci. Comp.} \textbf{22} 917--950

\bibitem{be} Beig R 1994 \emph{The classical theory of canonical general relativity (Canonical gravity: from classical to quantum)} ed Ehlers J and Friedrich H  (Berlin: Springer) pp 59–-80

\bibitem{bf} Brill F and Flaherty F 1976 \emph{Comm. Math. Phys.} \textbf{50} 157--165

\bibitem{cfm} Choquet  Y  Fischer  A   E and Marsden  J  E 1979
\emph{Maximal hypersurfaces and positivity of mass (Proc. of the Enrico Fermi Summer School of the Italian Physical Society)}  ed Ehlers J (Amsterdam: North-Holland) pp 396–-456

\bibitem{fmm} Fischer  A  Marsden  J and Moncrief  V 1980 \emph{Ann. Inst. H. Poincar\'e} \textbf{33}  147--194

\bibitem{gt} Gilbart  D and  Trudinger  N  S 2001 \emph{Elliptic Partial Differential Equations of Second Order}  (Berlin:  Springer-Verlag)

\bibitem{ge} Gerhardt C 1983  \emph{Comm. Math. Phys.} \textbf{89} 523--553

 \bibitem{go} Goddard J 1977 \emph{Math. Proc. Cambridge Phil. Soc.} \textbf{82}  489--495

\bibitem{lo}  L\'opez  R 2006 \emph{ Comm. Math. Phys.}  \textbf{266}   331--342

\bibitem{lo2}  L\'opez  R 2007 \emph{J. Geom. Phys. } \textbf{57}  2178--2186

\bibitem{lo3}  L\'opez  R 2008 \emph{Royal Soc. Edinburgh Proc. A (Math).} \textbf{138}  1067--1096

\bibitem{lo4}  L\'opez  R 2009 \emph{Osaka J. Math.} \textbf{46}  1--20  

\bibitem{mt}  Marsden J and Tipler F 1980  \emph{Phys. Rep.}  \textbf{66}    109--139

\bibitem{mo} Montiel S 1988 \emph{ Indiana Univ. Math. J.} \textbf{37} 909--917

\bibitem{on}  O'Neill  B 1983  \emph{Semi-Riemannian Geometry with applications to Relativity} (London: Academic Press)

\bibitem{pa} Pastor J A 2000 \emph{Class. Quantum Grav.}  \textbf{17} 1921--1934

\bibitem{Ratcliffe} Ratcliffe  J 2006 \emph{Foundations of hyperbolic manifolds} (New York: Springer-Verlag)

\bibitem{re} Rendall A D 1996 \emph{Helv. Phys. Acta} \textbf{69} 490--500

\bibitem{sy} Schoen R and Yau S T 1979 \emph{Comm. Math. Phys.} \textbf{65} 45--76

\bibitem{st} Stumbles S 1981 \emph{Ann. Physics.} \textbf{133} 28--56

\end{thebibliography}
\end{document}